\documentclass[reqno,10pt]{amsart}
\usepackage{amssymb}
\usepackage{latexsym}
\usepackage{hyperref}
\usepackage{amsmath}
\usepackage{amscd}
\usepackage{color}
\usepackage{enumerate}
\usepackage{amsfonts}
\usepackage{graphicx}
\usepackage{mathrsfs}
\usepackage{setspace}
\newcommand{\eps}{{\varepsilon}}
\theoremstyle{plain}
\newtheorem{theorem}{Theorem}
\newtheorem{proposition}[theorem]{Proposition}
\newtheorem{lemma}[theorem]{Lemma}
\newtheorem{corollary}[theorem]{Corollary}

\theoremstyle{definition}

\newtheorem{remark}[theorem]{Remark}

\numberwithin{equation}{section}
\numberwithin{theorem}{section}
\let\Re=\undefined\DeclareMathOperator*{\Re}{Re}
\let\Im=\undefined\DeclareMathOperator*{\Im}{Im}

\def\ge{\geqslant}
\def\le{\leqslant}
\def\b{\boldsymbol}
\begin{document}
\title[Scattering for NLS system in $\mathbb{R}^5$]{Scattering For Mass-resonance Nonlinear Schr\"odinger System in 5D}
\author[F. Meng]{Fanfei Meng}
\address{Graduate School of China Academy of Engineering Physics,  \ Beijing, \ China, \ 100089, }
\email{mengfanfei17@gscaep.ac.cn}
\author[C. Xu]{Chengbin Xu}
\address{Graduate School of China Academy of Engineering Physics,  \ Beijing, \ China, \ 100089, }
\email{xuchengbin19@gscaep.ac.cn}
\date{\today}
\maketitle

\begin{abstract}
In this paper, we simplify the proof of M. Hamano in \cite{Hamano2018}, scattering theory of the solution to \eqref{NLS system}, by using the method from B. Dodson and J. Murphy in \cite{Dodson2018}.
Firstly, we establish a criterion to ensure the solution scatters in $ H^1(\mathbb{R}^5) \times H^1(\mathbb{R}^5) $.
In order to verify the correctness of the condition in scattering criterion, we must exclude the concentration of mass near the origin.
The interaction Morawetz estimate and Galilean transform characterize a decay estimate, which implies that the mass of the system cannot be concentrated.
\end{abstract}

\maketitle
\section{Introduction}
We consider the quadratic nonlinear Schr\"odinger system:
\begin{equation}\tag{$\text{NLS system}$}
\left\{
\begin{aligned}
& i\partial_{t}\b{\rm u} + A \b{\rm u} + \b{\rm f} = \b{\rm 0}, \\
& \b{\rm u}(0,x)= \b{\rm u}_0(x),
\end{aligned}
\right.
\quad (t, x) \in \mathbb{R} \times \mathbb{R}^d.
\label{NLS system}
\end{equation}
where
\begin{equation*}
\b{\rm u} = \left(
\begin{aligned}
u \\
v
\end{aligned}
\right),~ A = \left(
\begin{aligned}
\Delta ~~~ & ~~~~ 0 \\
0 ~~~~ & ~~~~ \kappa\Delta
\end{aligned}
\right),~\b{\rm f} = \left(
\begin{aligned}
v\overline{u} \\
u^2
\end{aligned}
\right),~\b{\rm u_0} = \left(
\begin{aligned}
u_0 \\
v_0
\end{aligned}
\right),
\end{equation*}
$ u,v: \mathbb{R} \times \mathbb{R}^{d} \rightarrow \mathbb{C} $ are unknown functions and $ \Delta $ denotes the Laplacian in $ \mathbb{R}^{d} $.

From physical viewpoint, \eqref{NLS system} is related to the Raman amplification in a plasma.
This process is a nonlinear instability phenomenon (see \cite{Colin2009} for more detail).
Solutions to \eqref{NLS system} conserve the \emph{mass}, \emph{energy} and \emph{momentum}, defined respectively by
\begin{equation*}
\begin{aligned}
M(\b{\rm u}) & = \Vert u \Vert_{L^2}^2 + \Vert v \Vert_{L^2}^2 \equiv M(\b{\rm u}_{0}), \\
E(\b{\rm u}) & = H(\b{\rm u}) - R(\b{\rm u}) \equiv E(\b{\rm u}_{0}), \\
P(\b{\rm u}) & = \Im \int_{\mathbb{R}^d} \left( \overline{u}\nabla u + \frac{1}{2}\overline{v}\nabla v \right) {\rm d}x \equiv P(\b{\rm u}_0),
\end{aligned}
\end{equation*}
where
\[
\begin{aligned}
  & (\text{kinetic energy}) \quad & H(\b{\rm u}) & = \Vert \nabla u \Vert_{L^2}^2 + \frac{\kappa}{2} \Vert \nabla v \Vert_{L^2}^2, \\
  & (\text{potential energy}) \quad & R(\b{\rm u}) & = \Re \int_{\mathbb{R}^d} \overline{v}u^2 {\rm d}x.
\end{aligned}
\]

The equation \eqref{NLS system} is invariant under the scaling
\[
\b{\rm u}_{\lambda}(t, x) = \lambda^2 \b{\rm u} \left( \lambda^2 t, \lambda x \right)
\]
for $ \lambda > 0 $.
The critical regularity of Sobolev space is $ {\rm \dot{H}}^{s_c} $,
i.e. $ \Vert \b{\rm u}_{\lambda} \Vert_{{\rm \dot{H}}_{x}^{s_c}(\mathbb{R}^d)} =  \Vert \b{\rm u} \Vert_{{\rm \dot{H}}_{x}^{s_c}(\mathbb{R}^d)} $ where $ s_c = \frac{d}{2} - 2 $.
Therefore, the equation \eqref{NLS system} is called mass-subcritical if $ d \le 3 $, mass-critical if $ d=4 $, energy-subcritical if $ d=5 $, and energy-critical if $ d=6 $.
Besides, $ \kappa = \frac{1}{2} $ is called the mass-resonance condition, in which case \eqref{NLS system} has the Galilean invariance:
\[
\left(
\begin{aligned}
u(t, x) \\
v(t, x)
\end{aligned}
\right) \to \left(
\begin{aligned}
e^{ix\cdot\xi}e^{-t\vert\xi\vert^2}u(t, x-2t\xi) \\
e^{2ix\cdot\xi}e^{-2t\vert\xi\vert^2}v(t, x-2t\xi)
\end{aligned}
\right)
\]
for any $ \xi \in \mathbb{R}^d $, while \eqref{NLS system} does not have this invariance as long as $ \kappa \ne \frac{1}{2} $.

Unlike the general nonlinear Schr\"odinger equation:
\[\tag{$\text{NLS}$}
i\partial_{t}u + \Delta u + \lambda \vert u \vert^{p-1} u = 0, \quad (t, x) \in \mathbb{R} \times \mathbb{R}^d,
\label{NLS}
\]
\eqref{NLS system} only has the focusing case.
Both focusing case ($ \lambda > 0 $) and defocusing case ($ \lambda < 0 $) of \eqref{NLS} have been studied in a large amount of literature,
such as \cite{Bourgain1999, Colliander2008, Dodson2012, Dodson, Dodson20161, Dodson2017, Duyckaerts2007, Killip2010, Miao2013, Miao2014, Miao2017, Visan2007}
by B. Dodson, T. Duyckaerts, R. Killip, C. Miao, M. Visan and so on.
Unlike the system with symmetric interaction:
\[\tag{$\text{sNLS system}$}
\left\{
\begin{aligned}
i\partial_{t} u + \Delta u + ( \vert u \vert^2 + \vert v \vert^2)u = 0, \\
i\partial_{t} v + \Delta v + ( \vert u \vert^2 + \vert v \vert^2)v = 0,
\end{aligned}
\right.
\quad (t, x) \in \mathbb{R} \times \mathbb{R}^d,
\label{sNLS system}
\]
\eqref{NLS system} contains two different status unknown functions.
The \eqref{sNLS system} in 3d has been studied by S. Xia, C. Xu \cite{Xia2019} and G. Xu \cite{Xu2014}.

In this paper, we consider energy-subcritical case of \eqref{NLS system} under the mass-resonance condition, that is, $ (\kappa, d) = (\frac{1}{2}, 5) $.
By solution, we mean a function $ \b{\rm u} \in C_{t}(I, {\rm H}_{x}^1(\mathbb{R}^5)) $ on an interval $ I \ni 0 $ satisfying the Duhamel formula
\[
\b{\rm u}(t) = \mathcal{S}(t) \b{\rm u}_0 + i \int_{0}^{t} \mathcal{S}(t-s) \b{\rm f}(s) {\rm d}s
\]
for $ t \in I $,
where we denote the Schr\"odinger group $ \mathcal{S}(t)\b{\rm w} =: \left( e^{it\Delta}w_1, e^{\kappa it\Delta}w_2 \right) $ for any $ \b{\rm w} =: (w_1, w_2) $.
In order to study the local theory of the Cauchy problem for \eqref{NLS system}, we need the Strichartz estimate.
By using the standard contraction mapping theorem, we can show: $ \exists ~ \delta_0 > 0 $ such that, if
\begin{equation*}
\left\Vert \mathcal{S}(t) \b{\rm u}_0 \right\Vert_{{\rm L}_{t}^6(I, {\rm L}_{x}^3(\mathbb{R}^5))} < \delta_0,
\end{equation*}
then we have a unique solution $ \b{\rm u}(t) = (u(t), v(t)) $ in the interval $ I $.
For large data we have solution $ \b{\rm u}(t) $ with a maximal interval of existence $ I_{\max} = (T_-(\b{\rm u}), T_+(\b{\rm u})) $.
If $ I_{\max} = \mathbb{R} $, we call the solution is global.
A global solution $ \b{\rm u} $ ``scatters'', i.e.
there exists $ \b{\rm u}_+ \in {\rm H}^1(\mathbb{R}^5) $ such that
\begin{equation*}
\lim_{t \to \pm \infty} \left\Vert \b{\rm u}(t) - \mathcal{S}(t) \b{\rm u}_+ \right\Vert_{{\rm H}^1(\mathbb{R}^5)} = 0.
\end{equation*}
The \eqref{NLS system} admits a global but nonscattering solution
\[
(u(t, x), v(t, x)) = (e^{it}\phi(x), e^{2it}\varphi(x)),
\]
where $ (\phi, \varphi) \neq {\b 0} $ is a non-negative radial solution to the elliptic system
\[
\left\{
\begin{aligned}
& \phi - \Delta \phi = \phi \varphi, \\
& 2\varphi - \kappa\Delta \varphi = \phi^2,
\end{aligned}
\right.
~~~~~~ x \in \mathbb{R}^5.
\]
We call $ \b{\rm Q} = (\phi, \varphi) $ the ``{\bf ground state}'' which is the one of smallest energy.
In \cite{Hamano2018}, M. Hamano determined the global behavior of the solutions to the system with data below the ground state and proved a blowing-up result if the data had finite variance or was radial.

Our main result in this paper is follows:
\begin{theorem}\label{main}
For $ (\kappa, d) = (\frac{1}{2}, 5) $ in the \eqref{NLS system}.
If the initial data $ \b{\rm u}_0 \in {\rm H}^1(\mathbb{R}^5) $ satisfies
$ M( \b{\rm u}_0 ) E( \b{\rm u}_0 ) < M( \b{\rm Q} ) E( \b{\rm Q} ) $ and $ M(\b{\rm u}_0) H(\b{\rm u}_0) \le M(\b{\rm Q}) H(\b{\rm Q}) $,
then the solution of \eqref{NLS system} is global and scatters in $ {\rm H}^1 $.
\end{theorem}

Theorem \ref{main} was originally proven by M. Hamano in \cite{Hamano2018} for $ \kappa = \frac{1}{2} $,
through the use of concentration compactness by Kenig-Merle in \cite{Kenig2006}.
Later, using concentration compactness again, M. Hamano, T. Inui and K. Nishimura in \cite{Hamano2019} studied the scattering for $ \kappa > 0 $ and $ \b{\rm u}_0 $ radial.
We present a simplified proof here for the non-radial case under the mass-resonance condition.

\begin{remark}
In fact, the results of Theorem \ref{main} essentially holds for any $ \kappa > 0 $ if the initial data $ \b{\rm u}_0 $ in radial.
For $ \kappa = \frac{1}{2} $, we are able to use the Galilean invariance of \eqref{NLS system} to build interaction Morawetz estimate successfully without the radial assumption for $ \b{\rm u}_0 $.
But for $ \kappa \ne \frac{1}{2} $, we have to add the radial assumption owing to the lack of Galilean invariance when we study the scattering for \eqref{NLS system}.
Under this circumstance, the radial assumption for $ \b{\rm u}_0 $ implies the mass of the \eqref{NLS system}, if concentracted, must be at the origin.
Thus, we only need to establish a simpler Morawetz estimate instead of Proposition \ref{decay} to verify the condition of scattering criterion.
\end{remark}

Our proof of Theorem \ref{main} consists of two steps:

Firstly, we establish a scattering criterion as follows, using the method from B. Dodson, J. Murphy \cite{Dodson2016} and T. Tao \cite{Tao2004}.
\begin{proposition}[Scattering criterion]\label{scat}
Let $ \b{\rm u}_0, \b{\rm Q} $ be as in Theorem \ref{main} and suppose further that $ \Vert \b{\rm u}_0 \Vert_{{\rm H}^1(\mathbb{R}^5)} \lesssim E_0 $.
Let $ \b{\rm u} : \mathbb{R} \times \mathbb{R}^5 \to \mathbb{C} $ be the corresponding global solution to \eqref{NLS system}.
Suppose that $ \exists ~ t_0 \in I \subset \mathbb{R} $, an arbitrary interval of length $ T_0 $, such that
\begin{equation*}
\Vert \b{\rm u} \Vert_{{\rm L}_{t}^6([t_0-l, t_0], {\rm L}_{x}^3(\mathbb{R}^5))} \le \eps^{\frac{1}{30}},
\end{equation*}
where $ \eps = \eps(E_0) $ is sufficiently small, $ T_0 = T_0(\eps, E_0) $ is large enough and $ l = \eps^{-\frac{4}{5}}$.
Then $ \b{\rm u} $ scatters forward in time.
\end{proposition}

Secondly, in order to verify the condition of the above criterion,
we prove a certain decay estimate, which can be deduced from an interaction Morawetz estimate.
The proof of the following interaction Morawetz estimate relies on a Galilean invariance of \eqref{NLS system}.
\begin{proposition}[Interaction Morawetz estimate]\label{decay}
Let $ \b{\rm u}_0, \b{\rm Q}, I $ be as in Theorem \ref{scat}, and suppose further that
\begin{equation}
M(\b{\rm u}_0) = E(\b{\rm u}_0) = E_0.
\label{E0}
\end{equation}
Let $ \b{\rm u} : \mathbb{R} \times \mathbb{R}^5 \to \mathbb{C} $ be the corresponding global solution to \eqref{NLS system}.
Then there exists $ \delta > 0 $ such that for $ R_0 = R_0(\delta, M(\b{\rm u}), \b{\rm Q}) $ sufficiently large,
\begin{equation}
\frac{\delta}{JT_0} \int_{I} \int_{R_0}^{R_0e^{J}}\frac{1}{R^5} \int_{\mathbb{R}^5} \int_{\mathbb{R}^5} \int_{\mathbb{R}^5} L^{\xi}_{\kappa} \Gamma^{2} \left( \frac{x - s}{R} \right) \Gamma^{2} \left( \frac{y - s}{R} \right) N_{\kappa} {\rm d}x {\rm d}y {\rm d}s \frac{{\rm d}R}{R} {\rm d}t \lesssim \nu E_0^2,
\label{Me}
\end{equation}
for $ \nu = \frac{R_0e^{J}}{JT_0} + \eps $, where $ 0 < \eps < 1, J \ge \eps^{-1} $ are both constant and $ L^{\xi}_{\kappa} = \left( 2 \vert \nabla u^{\xi}(x) \vert^2 + \kappa \vert \nabla v^{\xi}(x) \vert^2 \right), N_{\kappa} = \left( 2\kappa \vert u(y) \vert^2 + \vert v(y) \vert^2 \right) $,
$ \Gamma \in C^{\infty} $ is a radial decreasing function satisfying
\begin{equation}
\Gamma(x) = \left\{
\begin{aligned}
1 & \quad \quad \vert x \vert \le 1 - \eps, \\
0 & \quad \quad \vert x \vert > 1,
\end{aligned}
\right.
\label{Gamma}
\end{equation}
and $ \b{\rm u}^{\xi} = (e^{\kappa ix\cdot\xi}u, e^{ix\cdot\xi}v) $ with
\begin{equation}
\xi = - \frac{\int_{\mathbb{R}^5} \Im \left( 2\kappa u(x) \overline{\nabla u(x)} + \kappa v(x) \overline{\nabla v(x)} \right) \Gamma^{2} \left(\frac{x-s}{R}\right) {\rm d}x}{\int_{\mathbb{R}^5} \left(2\kappa \vert u(x) \vert^2 + \vert v(x) \vert^2 \right) \Gamma^{2} \left(\frac{x-s}{R}\right){\rm d}x}
\label{xi}
\end{equation}
(unless the denominator is zero, in which case $ \xi = \xi(t, s, R) = 0 $).
\end{proposition}

Combining the two steps, that are respectively formulated as above, we can obtain Theorem \ref{main}.
\begin{remark}
{\bf Where is the mass-resonance condition specifically used? }
After computing $ \frac{\rm d}{{\rm d}t} M(t) $ carefully in the proof of Proposition \ref{decay},
we find that $ \mathcal{C + E} $ in \eqref{C+E} always stays the same under Galilean transformation $ \b{\rm u}^{\xi} = (e^{\kappa ix\cdot\xi}u, e^{ix\cdot\xi}v) $ for any $ \kappa > 0 $.
The condition $ \kappa = \frac{1}{2} $ is used to match the linear terms and nonlinear term in \eqref{NLS system}.
Therefore, $ \kappa $ reflects the coupling effect of this system.
Only when $ \kappa = \frac{1}{2} $ can we deduce the coercivity (Lemma \ref{coer2}, Lemma \ref{coerb}),
which is necessary to bound the major term of the interaction Morawetz estimate.
\end{remark}

\subsection{Outline of the paper}
The organization of this paper is as follows.
In Section 2, we clarify some preliminaries including notations and basic results.
In addition, we give some properties of ground state, based on which we establish a series of coercivity results.
We prove the scattering criterion and interaction Morawetz estimate in Section 3 and 4, respectively.
Finally, in Section 5, we use the results of Proposition \ref{decay} and \ref{scat} to complete the proof of main theorem.

\section{Preliminaries}
We mark $ A \lesssim B $ to mean there exists a constant $ C > 0 $ such that $ A \leqslant C B $.
We indicate dependence on parameters via subscripts, e.g. $ A \lesssim_{x} B $ indicates $ A \leqslant CB $ for some $ C = C(x) > 0 $.
For $ 1 \le p \le \infty $, we use $ p^\prime $ to denote the H\"older conjugate index of $ p $ with $ \frac{1}{p} + \frac{1}{p^\prime} = 1 $.
We write $ L_{t}^{q}L_{x}^{r} $ to denote the Banach space with norm
\begin{equation*}
\Vert u \Vert_{L_{t}^{q}(\mathbb{R}, L_{x}^{r}(\mathbb{R}^5))} := \left( \int_{\mathbb{R}} \left( \int_{\mathbb{R}^5} \vert u(t, x) \vert^{r} \mathrm{d}x \right) ^{\tfrac{q}{r}} \mathrm{d}t \right)^{\tfrac{1}{q}},
\end{equation*}
with the usual modifications when $ q $ or $ r $ are equal to infinity,
or when the domain $ \mathbb{R} \times \mathbb{R}^5 $ is replaced by space-time slab such as $ I \times \mathbb{R}^5 $.
We use $ (q, r) \in \Lambda_{s} $ to denote $ q \geqslant 2 $ and the pair satisfying
\begin{equation}
\frac{2}{q} = 5 \left( \frac{1}{2} - \frac{1}{r} \right) - s.
\label{Lambda}
\end{equation}
Lastly, to fit this artical, we use the notation $ {\rm L}^{q} $ to denote $ L^{q} \times L^{q} $ and $ {\rm H}^s $ to denote $ H^s \times H^s $.

\subsection{Variational characterization}
In this section, we are in the position to give the variational characterization for the sharp Gargliardo-Nirenberg inequality.
Firstly, We  will show the existence of the Ground state.
As a corollary, we obtain  the sharp Gargliardo-Nirenberg inequality.
Then we use the properties of the ground state to establish the Coercivity condition which will be used in the proof of Morawetz Estimate \eqref{Me}.

\begin{proposition}
[The existence of ground state, \cite{Hamano2018}]\label{ground state}
The minimal $J_{\min}$ of the nonnegative funtional
\[
    J(\b{\rm u}) : = (M(\b{\rm u}))^{\frac{1}{2}}(H(\b{\rm u}))^{\frac{5}{2}}(R(\b{\rm u}))^{-2},  \qquad  \b{\rm u} \in {\rm H}^{1}(\mathbb R^{5}) \setminus \{{\b 0}\}
\]
are attained at $ {\b{\rm u}} = (u, v) $,
whose expression has to be in the form of $ (u, v) = (e^{i\theta_1} m \phi(nx), e^{i\theta_2} m \varphi(nx)) $,
where $ m, n > 0 $, $ \theta_1, \theta_2 \in \mathbb R $,
and $ \b{\rm Q} = (\phi, \varphi) \neq {\b 0} $ is the non-negative radial solution of the equation
\[
\left\{
\begin{aligned}
& \phi - \Delta \phi = \phi \varphi, \\
& 2\varphi - \kappa\Delta \varphi = \phi^2,
\end{aligned}
\right.
~~~~~~ x \in \mathbb{R}^5.
\]
$ {\b{\rm Q}} $ is called a {\bf ground state} with $ J({\b{\rm Q}}) = J_{\min} $. The sets of all ground states is denotes as $ \mathcal{G} $.
All ground states share the same mass, denoted as $ M_{gs} $.
\end{proposition}

\begin{remark}
For ground state $ \b{\rm Q} = (\phi, \varphi) $, we have the following scaling identity
\begin{align*}
& M \left( \lambda^{\alpha} \b{\rm Q}(\lambda^{\beta}\cdot) \right) + E \left( \lambda^{\alpha} \b{\rm Q}(\lambda^{\beta}\cdot) \right) \\
= & \lambda^{2\alpha - 5\beta}M(\b{\rm Q}) + \lambda^{2\alpha - 3\beta}H(\b{\rm Q}) - \lambda^{3\alpha - 5\beta}R(\b{\rm Q}) \quad \forall \lambda \in (0, \infty).
\end{align*}
Using variational derivatives and letting  $ \lambda = 1 $ in both sides of above identity, we can obtain
\[
\begin{aligned}
0 = & \Re \int_{\mathbb{R}^5} \left( 2\phi - 2\Delta \phi + 2\phi \varphi \right) \cdot \overline{ \left( \frac{\rm d}{{\rm d}\lambda} \Big|_{\lambda=1} \lambda^{\alpha} \phi(\lambda^{\beta}x) \right) } \\
& \quad \quad + \left( 2\varphi - \kappa \Delta \varphi + \phi^2 \right) \cdot \overline{ \left( \frac{\rm d}{{\rm d}\lambda} \Big|_{\lambda=1} \lambda^{\alpha} \varphi(\lambda^{\beta}x) \right) } {\rm d}x   \\
= & (2\alpha - 5\beta)M(\b{\rm Q}) + (2\alpha - 3\beta)H(\b{\rm Q}) - (3\alpha - 5\beta)R(\b{\rm Q}) \\
= & (2M + 2H - 3R)\alpha - (5M + 3H -5R)\beta, \quad \forall \alpha, \beta \in \mathbb{R}.
\end{aligned}
\]
This yields
\begin{equation}
M(\b{\rm Q}) : H(\b{\rm Q}) : R(\b{\rm Q}) = 1 : 5 : 4.
\label{Q}
\end{equation}
\end{remark}

Using the above proposition, we can directly obtain the following corollary.
\begin{corollary}[Gagliardo-Nirenberg inequality]
\begin{equation}
R(\b{\rm u}) \leqslant C_{GN} [M(\b{\rm u})]^{\frac{1}{4}} [H(\b{\rm u})]^{\frac{5}{4}},
\label{GN-inequality}
\end{equation}
where $ C_{GN} = 4 \cdot 5^{-\frac{5}{4}} M_{gs}^{-\frac{1}{2}} $.
The equality holds if and only if $ \b{\rm u} \in {\rm H}^{1}(\mathbb R^{5}) $ is a minimal element of functional $ J(\b{\rm u}) $,
that is to say $ \b{\rm u} \in \mathcal G $, or $ \b{\rm u} = {\b 0} $.
\end{corollary}

\subsection{Some useful inequalities}\label{ineq}
In this subsection, we show some important inequalities which are will be used frequently in the following sections.

Denote the free Schr\"odinger group to $ e^{it\Delta} $, that is,
\begin{equation}
\left( e^{it\Delta} f \right) (x) = \frac{1}{(4\pi it)^{d/2}} \int_{\mathbb{R}^d} e^{\frac{i\vert x-y \vert^2}{4t}} f(y) {\rm d}y, \quad t \ne 0.
\label{group}
\end{equation}
It is easy to deduce the dispersive estimate
\begin{equation*}
\Vert e^{it\Delta} f \Vert_{L_{x}^{\infty}(\mathbb{R}^d)} \lesssim \vert t \vert^{-\frac{d}{2}} \Vert f \Vert_{L_{x}^{1}(\mathbb{R}^d)}, \quad t \ne 0,
\end{equation*}
which together with $ \Vert e^{it\Delta} f \Vert_{L_{x}^{2}(\mathbb{R}^d)} \equiv \Vert f \Vert_{L_{x}^{2}(\mathbb{R}^d)} $ yields by interpolation theorem,
\begin{equation}
\Vert e^{it\Delta} f \Vert_{L_{x}^{r}(\mathbb{R}^d)} \lesssim \vert t \vert^{-d(\frac{1}{2}-\frac{1}{r})} \Vert f \Vert_{L_{x}^{r^\prime}(\mathbb{R}^d)}, \quad t \ne 0,
\label{disp}
\end{equation}
for $ 2 \le r < \infty $.

Strichartz estimates for the propagator $e^{it\Delta}$ have been proved in \cite{Ginibre1992} and \cite{Strichartz1977}.
Combining these with the Christ--Kiselev lemma \cite{CK} and the endpoint case in \cite{Keel1998}, we arrive at the following by $ TT^{\ast} $ method:

\begin{theorem}[Strichartz estimates]\label{strichartz}
The solution $ u $ to
\[
iu_{t} + \Delta u = f
\]
on an interval $ I \ni t_{0} $ obeys
\begin{equation}
\Vert u \Vert_{L^{q}_{t}L^{r}_{x}(I \times \mathbb{R}^{d})} \le C \left( \Vert u(t_{0}) \Vert_{L^{2}(\mathbb{R}^{d})} + \Vert f \Vert_{L^{\tilde{q}^\prime}_{t}L^{\tilde{r}^\prime}_{x}(I \times \mathbb{R}^{d})} \right)
\label{Strichartz}
\end{equation}
whenever $ (q,r), (\tilde{q},\tilde{r}) \in \Lambda_{0} $ as in \eqref{Lambda}, $ 2 \le q, \tilde{q} \le \infty $, and $ q \neq \tilde{q} $.
\end{theorem}

Thanks to \eqref{Q}, which shows us the relation between the mass and the energy of ground state,
we can find the lower bound of the energy in local.
\begin{lemma}[Coercivity $I$]\label{coer1}
For $ (\kappa, d) = (\frac{1}{2}, 5) $ in the \eqref{NLS system}.
If $ M(\b{\rm u}_0) E(\b{\rm u}_0) \le (1 - \delta) M(\b{\rm Q}) E(\b{\rm Q}) $
and $ M(\b{\rm u}_0) H(\b{\rm u}_0) < M(\b{\rm Q}) H(\b{\rm Q}) $, then there exists $ \delta^\prime = \delta^\prime (\delta) > 0 $ so that
\begin{equation}
M(\b{\rm u}) H(\b{\rm u}) \le (1 - \delta^\prime) M(\b{\rm Q}) H(\b{\rm Q})
\label{coe1}
\end{equation}
for all $ t \in I $,  where $ \b{\rm u} : I \times \mathbb{R}^5 \to \mathbb{C}^2 $ is the maximal-lifespan solution to \eqref{NLS system}.
In particular, $ I = \mathbb{R} $ and $ \b{\rm u} $ is uniformly bounded in $ {\rm H}^1(\mathbb{R}^5) $.
\end{lemma}

\begin{remark}\label{con}
Using \eqref{Q}, \eqref{GN-inequality}, \eqref{coe1}, and the conservation of mass and energy, we have
\begin{equation*}
1 - \delta \ge \frac{M(\b{\rm u}) E(\b{\rm u})}{M(\b{\rm Q}) E(\b{\rm Q})} \ge 5 y - 4 y^{\frac{5}{4}},
\end{equation*}
where
\begin{equation*}
y(t) = \frac{M(\b{\rm u}) H(\b{\rm u})}{ M(\b{\rm Q}) H(\b{\rm Q})} \in C(I).
\end{equation*}
Taking into account the continuity of $ y(t) $ and the case of $ t = 0 $,  we can easily \eqref{coe1} holds.

At the same time, the calculation above suggests $ y \ne 1 $, i.e. $ M(\b{\rm u}_0) H(\b{\rm u}_0) \ne M(\b{\rm Q}) H(\b{\rm Q}) $ as long as $ M(\b{\rm u}_0) E(\b{\rm u}_0) < M(\b{\rm Q}) E(\b{\rm Q}) $, because the highest point in the graph of $ 5 y - 4 y^{\frac{5}{4}} $ is $ (y, 5 y - 4 y^{\frac{5}{4}}) = (1, 1) $.
\end{remark}

\begin{remark}
In fact, under the conditions of Lemma \ref{coer1}, \eqref{coe1} holds for any $ t \in I $, the maximal lifespan of $ \b{\rm u}(t) $. 
In particular, $ H(\b{\rm u}) $ is bounded and hence the solution to \eqref{NLS system} $ \b{\rm u} $ is global. 
\end{remark}

\begin{lemma}
[Coercivity $II$]\label{coer2}
Suppose $ M(\b{\rm u}) H(\b{\rm u}) \le (1 - \delta) M(\b{\rm Q}) H(\b{\rm Q}) $,
then there exists $ \delta^\prime = \delta^\prime (\delta) > 0 $ such that
\begin{equation}
4 H(\b{\rm u}^{\xi}) - 5 R(\b{\rm u}) \ge \delta^{\prime} H(\b{\rm u}^{\xi}),
\label{coe2}
\end{equation}
where $ \b{\rm u}^{\xi} $ be as in Proposition \ref{decay}.
\end{lemma}

\begin{proof}
Using the fact that $ C_{GN} = 4 \cdot 5^{-\frac{5}{4}} M_{gs}^{-\frac{1}{2}} = 4 \cdot 5^{-\frac{5}{4}} \left[ M(\b{\rm Q}) \right]^{-\frac{1}{2}} = \frac{4}{5} \left[ M(\b{\rm Q}) H(\b{\rm Q}) \right]^{-\frac{1}{4}} $, we have
\begin{equation*}
\begin{aligned}
R(\b{\rm u}) \le C_{GN} \left[ M(\b{\rm u}) \right]^{\frac{1}{4}} \left[ H(\b{\rm u}) \right]^{\frac{5}{4}}
= \frac{4}{5} \left[ \frac{M(\b{\rm u}) H(\b{\rm u})}{M{\b{\rm Q}} H(\b{\rm Q})} \right]^{\frac{1}{4}} H(\b{\rm u}).
\end{aligned}
\end{equation*}

Owing to $ M(\b{\rm u}) = M(\b{\rm u}^{\xi}) $ and $ R(\b{\rm u}) = R(\b{\rm u}^{\xi}) $, we know furthermore
\begin{equation*}
\begin{aligned}
R(\b{\rm u}) = R(\b{\rm u}^{\xi}) \le & \frac{4}{5} \inf_{\xi \in \mathbb{R}^5} \left\{ \left[ \frac{M(\b{\rm u}) H(\b{\rm u}^{\xi})}{M{\b{\rm Q}} H(\b{\rm Q})} \right]^{\frac{1}{4}} H(\b{\rm u}^{\xi}) \right\} \\
\le & \frac{4}{5} \inf_{\xi \in \mathbb{R}^5} \left[ \frac{M(\b{\rm u}) H(\b{\rm u}^{\xi})}{M{\b{\rm Q}} H(\b{\rm Q})} \right]^{\frac{1}{4}} \inf_{\xi \in \mathbb{R}^5} H(\b{\rm u}^{\xi}) \\
\le & \frac{4}{5} (1 - \delta)^{\frac{1}{4}} H(\b{\rm u}^{\xi}),
\end{aligned}
\end{equation*}
which implies \eqref{coe2} holds.
\end{proof}

\begin{lemma}
[Coercivity on balls]\label{coerb}
There exists $ R = R(\delta,  M(\b{\rm u}),  \b{\rm Q}) > 0 $ sufficiently large such that
\begin{equation}
\sup_{t \in \mathbb{R}} M(\b{\rm u}^{\xi}_R) H(\b{\rm u}^{\xi}_R) < (1 - \delta) M(\b{\rm Q}) H(\b{\rm Q}),
\label{coe}
\end{equation}
where $ \b{\rm u}^{\xi}_R(x) = \chi_R(x) \b{\rm u}^{\xi}(x) $ for $ \chi_R $, a smooth cut function on $ B(0, R) \subset \mathbb{R}^5 $.
In particular, by Lemma \ref{coer2}, there exists $ \delta^\prime = \delta^\prime(\delta) > 0 $ so that
\begin{equation}
4 H(\b{\rm u}^{\xi}_R) - 5 R(\b{\rm u}_R) \ge \delta^{\prime} H(\b{\rm u}^{\xi}_R)
\label{coe3}
\end{equation}
uniformly for $ t \in \mathbb{R} $.
\end{lemma}

\begin{proof}
First note that multiplication by $ \chi $ only decreases the $ L^{2}(\mathbb{R}^5) $-norm, that is
\begin{equation*}
M \left( \chi_{R} \b{\rm u}^{\xi}(t) \right) \leqslant M \left( \b{\rm u}^{\xi}(t) \right)
\end{equation*}
uniformly for $ t \in \mathbb{R} $. Thus, it suffices to consider the $ \dot{\rm H}^1(\mathbb{R}^5) $-norm.
For this, we will make use of the following identity:
\begin{equation*}
\int_{\mathbb{R}^5} \chi_{R}^2 \vert \nabla u^{\xi} \vert^2 {\rm d}x = \int_{\mathbb{R}^5} \vert \nabla ( \chi_{R} u^{\xi} ) \vert^2 + \chi_{R} \Delta (\chi_{R}) \vert u^{\xi} \vert^2 {\rm d}x,
\end{equation*}
which can be obtained by a direct computation. In particular,
\begin{equation*}
H(\b{\rm u}^{\xi}_R) \le H(\b{\rm u}^{\xi}) + \mathcal{O}\left( \frac{1}{R^2} M(\b{\rm u}) \right).
\end{equation*}
Choosing $ R \gg 1 $ sufficiently large depending on $ \delta,  M(\b{\rm u}) $ and $ \b{\rm Q} $, the result follows.
\end{proof}

\section{Proof of scattering criterion}
In this section, we will follow the strategy in \cite{Dodson2018} to prove the scattering criterion (Proposition \ref{scat}).
Roughly speaking, it states that if in any large window of time there always exists an interval large enough on which the scattering norm is very small,
then the solution of \eqref{NLS system} has to scatter.
\begin{proposition}[Scattering criterion]
Let $ \b{\rm u}_0, \b{\rm Q}, I $ be as in Theorem \ref{decay}, and suppose further that $ \Vert \b{\rm u}_0 \Vert_{{\rm H}^1(\mathbb{R}^5)} \lesssim E_0 $.
Let $ \b{\rm u} : \mathbb{R} \times \mathbb{R}^5 \to \mathbb{C} $ be the corresponding global solution to \eqref{NLS system}.
Suppose that $ \exists ~ t_0 \in I $ such that
\begin{equation}
\Vert \b{\rm u} \Vert_{{\rm L}_{t}^6([t_0-l, t_0], {\rm L}_{x}^3(\mathbb{R}^5))} \le \eps^{\frac{1}{18}},
\label{b1}
\end{equation}
where $ \eps = \eps(E_0) $ is sufficiently small and $ T_0 = T_0(\eps, E_0) $ is large enough.
Then $ \b{\rm u} $ scatters forward in time.
\end{proposition}

\begin{proof}
The entire proof process is divided into two major steps.

{\bf Step 1.} A standard argument yields scattering if for $ T_0 $ large enough
\begin{equation}
\Vert \b{\rm u}(t, x) \Vert^{6}_{{\rm L}_{t}^6(\mathbb{R}, {\rm L}_{x}^3(\mathbb{R}^5))} \lesssim_{E_0} T_0.
\label{0}
\end{equation}
We begin by splitting $ \mathbb{R} $ into $ J = J(\epsilon, E_0) $ intervals $ I_{j} $ such that
\begin{equation}
\left\Vert \mathcal{S}(t) \b{\rm u}_0 \right\Vert^{6}_{{\rm L}_{t}^6(I_{j}, {\rm L}_{x}^3(\mathbb{R}^5))} \le \eps^{\frac{1}{24}}.
\label{linear}
\end{equation}
For those $ I_{j} $ with $ \vert I_{j} \vert \le 2 T_0 $,
\begin{equation}
\Vert \b{\rm u}(t, x) \Vert^{6}_{{\rm L}_{t}^6(\cup I_{j}, {\rm L}_{x}^3(\mathbb{R}^5))} \lesssim_{E_0} \sum \langle I_{j} \rangle \lesssim T_0.
\label{Ij}
\end{equation}
So, we only need to consider $ j $ such that $ \vert I_{j} \vert > 2 T_0 $.
Therefore we fix some $ I_{j} = (a_{j}, b_{j}) $ and choose $ t_0 \in (a_{j}, a_{j} + T_0) $.

Note that $ I_{j} = (a_{j}, t_0] \cup (t_0, b_{j}) $ and $ t_0 - a_{j} < T_0 $, similarly to \eqref{Ij} we have
\begin{equation*}
\begin{aligned}
\left\Vert \b{\rm u} \right\Vert^6_{{\rm L}_{t}^6(I_{j}, {\rm L}_{x}^3(\mathbb{R}^5))} = & \left\Vert \b{\rm u} \right\Vert^6_{{\rm L}_{t}^6((a_{j}, t_0], {\rm L}_{x}^3(\mathbb{R}^5))} + \left\Vert \b{\rm u} \right\Vert^6_{{\rm L}_{t}^6((t_0, b_{j}), {\rm L}_{x}^3(\mathbb{R}^5))} \\
\lesssim & ~ T_0 + \left\Vert \b{\rm u} \right\Vert^6_{{\rm L}_{t}^6((t_0, b_{j}), {\rm L}_{x}^3(\mathbb{R}^5))}.
\end{aligned}
\end{equation*}
We use Strichartz estimate to find
\begin{equation*}
\begin{aligned}
& \left\Vert \b{\rm u}(t, x) \right\Vert_{{\rm L}_{t}^6((t_0, b_{j}), {\rm L}_{x}^3(\mathbb{R}^5))} \\
\le & \left\Vert \mathcal{S}(t-t_0) \b{\rm u}(t_0) \right\Vert_{{\rm L}_{t}^6((t_0, b_{j}), {\rm L}_{x}^3(\mathbb{R}^5))} + \left\Vert \int_{t_0}^{t} \mathcal{S}(t-s) \b{\rm f}(s) {\rm d}s \right\Vert_{{\rm L}_{t}^6((t_0, b_{j}), {\rm L}_{x}^3(\mathbb{R}^5))} \\
\le & \left\Vert \mathcal{S}(t-t_0) \b{\rm u}(t_0) \right\Vert_{{\rm L}_{t}^6((t_0, b_{j}), {\rm L}_{x}^3(\mathbb{R}^5))} + \left\Vert \b{\rm f} \right\Vert_{{\rm L}_{t}^{\frac{6}{5}}((t_0, b_{j}), {\rm L}_{x}^{\frac{3}{2}}(\mathbb{R}^5))} \\
\le & \left\Vert \mathcal{S}(t-t_0) \b{\rm u}(t_0) \right\Vert_{{\rm L}_{t}^6((t_0, b_{j}), {\rm L}_{x}^3(\mathbb{R}^5))} + \left\Vert \b{\rm u} \right\Vert^2_{{\rm L}_{t}^{\frac{12}{5}}((t_0, b_{j}), {\rm L}_{x}^{3}(\mathbb{R}^5))} \\
\le & \left\Vert \mathcal{S}(t-t_0) \b{\rm u}(t_0) \right\Vert_{{\rm L}_{t}^6((t_0, b_{j}), {\rm L}_{x}^3(\mathbb{R}^5))} + T_0^{\frac{1}{2}} \left\Vert \b{\rm u} \right\Vert^2_{{\rm L}_{t}^{6}((t_0, b_{j}), {\rm L}_{x}^{3}(\mathbb{R}^5))}.
\end{aligned}
\end{equation*}
The continuity argument tells us if $ \left\Vert \mathcal{S}(t-t_0) \b{\rm u}(t_0) \right\Vert_{{\rm L}_{t}^6((t_0, b_{j}), {\rm L}_{x}^3(\mathbb{R}^5))} \le \eps^{\frac{1}{18}} $ then \eqref{0} holds.

Now, we turn to the following identity
\begin{equation*}
\mathcal{S}(t-t_0) \b{\rm u}(t_0) = \mathcal{S}(t) \b{\rm u}_0 + i \int_{0}^{t_0} \mathcal{S}(t-s) \b{\rm f}(s) {\rm d}s.
\end{equation*}
Combining this with \eqref{linear} then suffices to establish
\begin{equation}
\left\Vert \int_{0}^{t_0} {\mathcal S}(t-s)\b{\rm f}(s) {\rm d}s \right\Vert_{{\rm L}_{t}^6((t_0, b_{j}), {\rm L}_{x}^3(\mathbb{R}^5))} \lesssim_{E_0} \eps^{\frac{1}{18}}.
\end{equation}

{\bf Step 2.} To show a stronger fact that \eqref{b1} implies
\begin{equation}
\left\Vert \int_{0}^{t_0} {\mathcal S}(t-s)\b{\rm f}(s) {\rm d}s \right\Vert_{{\rm L}_{t}^6((t_0, \infty), {\rm L}_{x}^3(\mathbb{R}^5))} \lesssim_{E_0} \eps^{\frac{1}{18}}.
\label{b2}
\end{equation}
we do as follows,
\begin{equation*}
\begin{aligned}
\int_{0}^{t_0} {\mathcal S}(t-s)\b{\rm f}(s) {\rm d}s
= & \int_{0}^{t_0-l} {\mathcal S}(t-s)\b{\rm f}(s) {\rm d}s + \int_{t_0-l}^{t_0} \mathcal{S}(t-s)\b{\rm f}(s) {\rm d}s \\
= & : I + II.
\end{aligned}
\end{equation*}

On one hand, we transform the identity
\begin{equation*}
\b{\rm u}(t_0-l) = \mathcal{S}(t_0-l)\b{\rm u}_0 + i \int_{0}^{t_0-l} \mathcal{S}(t_0-l-s)\b{\rm f}(s) {\rm d}s
\end{equation*}
into
\begin{equation*}
i \int_{0}^{t_0-l} \mathcal{S}(t-s)\b{\rm f}(s) {\rm d}s = \mathcal{S}(t-t_0+l)\b{\rm u}(t_0-l) - \mathcal{S}(t)\b{\rm u}_0
\end{equation*}
and use dispersive estimates of Schr\"odinger group to deduce
\begin{equation}
\begin{aligned}
\left\Vert I \right\Vert_{{\rm L}_{t}^6((t_0, \infty), {\rm L}_{x}^3(\mathbb{R}^5))} & \le \left\Vert I \right\Vert^{\frac{3}{4}}_{{\rm L}_{t}^{18}((t_0, \infty), {\rm L}_{x}^{\frac{18}{5}}(\mathbb{R}^5))} \cdot \left\Vert I \right\Vert^{\frac{1}{4}}_{{\rm L}_{t}^2((t_0, \infty), {\rm L}_{x}^{\frac{10}{3}}(\mathbb{R}^5))} \\
& \le \left\Vert \int_{0}^{t_0-l} \left\Vert {\mathcal S}(t-s)\b{\rm f}(s) \right\Vert_{{\rm L}_{x}^{\frac{18}{5}}(\mathbb{R}^5)} {\rm d}s \right\Vert^{\frac{3}{4}}_{{\rm L}_{t}^{18}(t_0, \infty)} \\
& \quad \cdot \left( \Vert \b{\rm u}(t_0-l) \Vert_{{\rm L}_{x}^{2}(\mathbb{R}^5)} + \Vert \b{\rm u}_0 \Vert_{{\rm L}_{x}^{2}(\mathbb{R}^5)} \right)^{\frac{1}{4}} \\
& \lesssim_{E_0} \left\Vert \int_{0}^{t_0-l} \vert t-s \vert^{-\frac{10}{9}} \Vert \b{\rm f}(s) \Vert_{{\rm L}_{x}^{\frac{18}{13}}(\mathbb{R}^5)} {\rm d}s \right\Vert^{\frac{3}{4}}_{{\rm L}_{t}^{18}(t_0, \infty)} \\
& \lesssim_{E_0} \left\Vert \vert t-t_0+l \vert^{-\frac{1}{9}} \right\Vert^{\frac{3}{4}}_{{\rm L}_{t}^{18}(t_0, \infty)} \cdot \Vert \b{\rm u} \Vert^{\frac{3}{4}}_{{\rm L}_{t}^{\infty}([0, t_0-l], {\rm L}_{x}^{\frac{9}{2}}(\mathbb{R}^5))} \\
& \quad \cdot \Vert \b{\rm u} \Vert^{\frac{3}{4}}_{{\rm L}_{t}^{\infty}([0, t_0-l], {\rm L}_{x}^{2}(\mathbb{R}^5))} \\
& \lesssim_{E_0} ~ l^{-\frac{1}{24}} = \eps^{\frac{1}{30}}, \quad\quad\quad\quad \text{for} ~ l = \eps^{-\frac{4}{5}}.
\end{aligned}
\label{nonlinear1}
\end{equation}

On the other hand, by Sobolev embedding and Strichartz estimate \eqref{Strichartz},
\begin{equation}
\begin{aligned}
\left\Vert II \right\Vert_{{\rm L}_{t}^6((t_0, \infty), {\rm L}_{x}^3(\mathbb{R}^5))} \lesssim & \left\Vert \vert \nabla \vert^{\frac{1}{2}} \b{\rm f} \right\Vert_{{\rm L}_{t}^{\frac{3}{2}}((t_0-l, t_0), {\rm L}_{x}^{\frac{30}{19}}(\mathbb{R}^5))} \\
\lesssim & ~ \Vert \b{\rm u} \Vert_{{\rm L}_{t}^6((t_0-l, t_0), {\rm L}_{x}^3(\mathbb{R}^5))} \left\Vert \vert \nabla \vert^{\frac{1}{2}} \b{\rm u} \right\Vert_{{\rm L}_{t}^2((t_0-l, t_0), {\rm L}_{x}^{\frac{10}{3}}(\mathbb{R}^5))} \\
\lesssim & ~ \Vert \b{\rm u} \Vert_{{\rm L}_{t}^6((t_0-l, t_0), {\rm L}_{x}^3(\mathbb{R}^5))} \le \eps^{\frac{1}{18}},
\end{aligned}
\label{nonlinear2}
\end{equation}
where we have used
\begin{equation*}
\left\Vert \vert \nabla \vert^{\frac{1}{2}} \b{\rm u} \right\Vert_{{\rm L}_{t}^2((t_0-l, t_0), {\rm L}_{x}^{\frac{10}{3}}(\mathbb{R}^5))} \lesssim 1.
\end{equation*}

\eqref{nonlinear1} and \eqref{nonlinear2} suggest that \eqref{b2} is true, which complete the proof.
\end{proof}

\section{Interaction Morawetz estimate}
We are now in the position to prove the interaction Morawetz estimate \eqref{Me} holds.
As we all know, the decay estimate of the solution $ \b{\rm u} $ can be characterized by Morawetz estimate.

Firstly, we define a functional of $ \b{\rm u}(t, x) = (u(t, x), v(t, x)) $, the solution of \eqref{NLS system}.
\begin{equation}
\begin{aligned}
M(t) = & 2 \int_{\mathbb{R}^5} \int_{\mathbb{R}^5} \Im \left( 2 \overline{u(x)} \nabla u(x) + \overline{v(x)} \nabla v(x) \right) \cdot \nabla a(x-y) N_{\kappa} {\rm d}x {\rm d}y,
\end{aligned}
\label{M}
\end{equation}
where $ N_{\kappa} = 2\kappa \vert u(y) \vert^2 + \vert v(y) \vert^2 $ and $ a \in C^{\infty} $ is a real function to be chosen later.

\begin{remark}
Compared with the classical \eqref{NLS} case in \cite{Dodson2018}, the coefficients in the definition expression of \eqref{M} are chosen carefully here.
For computing $ \frac{\rm d}{{\rm d}t} M(t) $, we need to use the equation \eqref{NLS system} to change the derivative of $ \b{\rm u} $ versus $ t $
into the derivative of $ \b{\rm u} $ versus $ x $ and the nonlinear term $ \b{\rm f} $.
And the chain rule of derivatives produces many terms.
That the ratio of the two coefficients of $ \overline{u(x)} \nabla u(x) $ and $ \overline{v(x)} \nabla v(x) $ is $ 2 : 1 $
is useful to get $ \Re \left( \overline{v(x)} u^2(x) \right) $ in \eqref{M1}.
In fact, $ \Re \left( \overline{v(x)} u^2(x) \right) $ is the final result after the positive and negative offsets corresponding to the nonlinear term $ \b{\rm f} $.
Besides, the exact ratio $ 2\kappa : 1 $ from the two coefficients of $ N_{\kappa} $ in \eqref{M} is used in the Cauchy-Schwartz inequality $ \eqref{CS} $,
which plays a vital role in estimating $ \mathcal{D+F} \ge 0 $.
\end{remark}

Let $ R \gg 1 $ be sufficiently large and let $ \phi $ and $ \phi_1 $ both be radial satisfying
\begin{equation*}
\phi(x) = \frac{1}{\omega_{5}R^{5}} \int_{\mathbb{R}^5} \Gamma^2 \left(\frac{x - s}{R}\right) \Gamma^2 \left(\frac{s}{R}\right) {\rm d}s,
\end{equation*}
and
\begin{equation*}
\phi_1(x) = \frac{1}{\omega_{5}R^{5}} \int_{\mathbb{R}^5} \Gamma^3 \left(\frac{x - s}{R}\right) \Gamma^2 \left(\frac{s}{R}\right) {\rm d}s,
\end{equation*}
where $ \omega_{5} $ is the volume of unit ball in $ \mathbb{R}^5 $ and $ \Gamma $ be as in \eqref{Gamma}.
Finally, we define
\begin{equation*}
\psi(x) = \frac{1}{\vert x \vert} \int_{0}^{\vert x \vert} \phi(r) {\rm d}r, \quad a(x) = \int_{0}^{\vert x \vert} \psi(r) r {\rm d}r.
\end{equation*}

\textbf{The proof of Proposition \ref{decay}:}
We rely on the equation \eqref{NLS system} to change $ \b{\rm u}_t $ equally into the derivative of $ \b{\rm u} $ with respect to the space variable $ x $.
Then we have
\begin{equation}
\begin{aligned}
\frac{\mathrm{d}}{\mathrm{d}t}M(t)
= & - \int_{\mathbb{R}^5} \int_{\mathbb{R}^5} \Delta W_{\kappa} \Delta a(x-y) N_{\kappa} {\rm d}x {\rm d}y \\
& -2 \int_{\mathbb{R}^5} \int_{\mathbb{R}^5} \Re \left( \overline{v(x)} u^2(x) \right) \Delta a(x-y) N_{\kappa} {\rm d}x {\rm d}y \\
& +4 \sum_{k=1}^{5} \sum_{j=1}^{5} \int_{\mathbb{R}^5} \int_{\mathbb{R}^5} R^{jk}_{\kappa} a_{jk}(x-y) N_{\kappa} {\rm d}x {\rm d}y \\
& -4 \sum_{k=1}^{5} \sum_{j=1}^{5} \int_{\mathbb{R}^5} \int_{\mathbb{R}^5} A^{j} a_{jk}(x-y) B^{k}_{\kappa} {\rm d}x {\rm d}y,
\end{aligned}
\label{M1}
\end{equation}
where $ W_{\kappa} = 2 \vert u(x) \vert^2 + \kappa \vert v(x) \vert^2, R^{jk}_{\kappa} = \Re \left( 2 u_{j}(x) \overline{u_{k}(x)} + \kappa v_{j}(x) \overline{v_{k}(x)} \right) $,
and $ A^{j} = \Im \left( 2 u(x) \overline{u_{j}(x)} + v(x) \overline{v_{j}(x)} \right), B^{k}_{\kappa} = \Im \left( 2\kappa u(y) \overline{u_{k}(y)} + \kappa v(y) \overline{v_{k}(y)} \right) $
if we use $ \partial_l $ to denote the partial differential respected to $ x_l $ for $ l \in \{ 1, 2, 3, 4, 5 \} $.

Direct computations yield $ \Delta a = 4 \psi + \phi $ and $ a_{jk} = \delta_{jk} \phi + P_{jk} (\psi - \phi) $,
where $ P_{jk}(x) := \delta_{jk} - \frac{x_{j}x_{k}}{\vert x \vert^2} $ and $ \psi - \phi \ge 0 $.
Due to the facts above, we have
\begin{equation}
\begin{aligned}
\frac{\rm d}{{\rm d}t}M(t)
= & \int_{\mathbb{R}^5} \int_{\mathbb{R}^5} \Delta W_{\kappa} \left( 4 \psi(x-y) + \phi(x-y) \right) N_{\kappa} {\rm d}x {\rm d}y \\
& + 2 \int_{\mathbb{R}^5} \int_{\mathbb{R}^5} \Re \left( \overline{v(x)} u^2(x) \right) \left( 4 \psi(x-y) + \phi(x-y) \right) N_{\kappa} {\rm d}x {\rm d}y \\
& + 4 \int_{\mathbb{R}^5} \int_{\mathbb{R}^5} L_{\kappa} \phi(x-y) N_{\kappa} {\rm d}x {\rm d}y \\
& + 4 \sum_{k=1}^{5} \sum_{j=1}^{5} \int_{\mathbb{R}^5} \int_{\mathbb{R}^5} R^{jk}_{\kappa} P_{jk}(x-y) (\psi(x-y) - \phi(x-y)) N_{\kappa} {\rm d}x {\rm d}y \\
& - 4 \int_{\mathbb{R}^5} \int_{\mathbb{R}^5} A \phi(x-y) B_{\kappa} {\rm d}x {\rm d}y \\
& - 4 \sum_{k=1}^{5} \sum_{j=1}^{5} \int_{\mathbb{R}^5} \int_{\mathbb{R}^5} A^{j} P_{jk}(x-y) (\psi(x-y) - \phi(x-y)) B^{k}_{\kappa} {\rm d}x {\rm d}y \\
=: & \mathcal{A+B+C+D+E+F},
\end{aligned}
\end{equation}
where $ A = \Im \left( 2 u(x) \overline{\nabla u(x)} + v(x) \overline{\nabla v(x)} \right), B_{\kappa} = \Im \left( 2\kappa u(y) \overline{\nabla u(y)} + \kappa v(y) \overline{\nabla v(y)} \right) $
and $ L_{\kappa} = 2 \vert \nabla u(x) \vert^2 + \kappa \vert \nabla v(x) \vert^2 $.

$ \mathcal{A} $ remains itself unchanged because it will be treated as an error term below.

As for $ \mathcal{B} $, we make use of the decomposition identity $ 4 \psi + \phi = 5 \phi_{1} + 4 (\psi - \phi) + 5 (\phi - \phi_{1}) $ and deduce
\begin{equation}
\begin{aligned}
\mathcal{B} = & -\frac{10}{\omega_{5}R^5} \int_{\mathbb{R}^5} \int_{\mathbb{R}^5} \int_{\mathbb{R}^5} \Re \left( \overline{v(x)} u^2(x) \right) \Gamma^{3} \left( \frac{x - s}{R} \right) \Gamma^{2} \left( \frac{y - s}{R} \right) N_{\kappa} {\rm d}x {\rm d}y {\rm d}s \\
& -8 \int_{\mathbb{R}^5} \int_{\mathbb{R}^5} \Re \left( \overline{v(x)} u^2(x) \right) \left( \psi(x-y) - \phi(x-y) \right) N_{\kappa} {\rm d}x {\rm d}y \\
& -10 \int_{\mathbb{R}^5} \int_{\mathbb{R}^5} \Re \left( \overline{v(x)} u^2(x) \right) \left( \phi(x-y) - \phi_{1}(x, y) \right) N_{\kappa} {\rm d}x {\rm d}y.
\end{aligned}
\end{equation}

We claim the quantity of $ \mathcal{C + E} $ is Galilean invariant, that is,
invariant under the the transformation
\begin{equation*}
\b{\rm u}(t, x) \mapsto \b{\rm u^{\xi}} = ( u^{\xi}(t, x), v^{\xi}(t, x) ) := ( e^{\kappa ix\cdot\xi}u, e^{ix\cdot\xi}v ), \quad\quad \forall \b{\rm u} = (u, v),
\end{equation*}
for any $ \xi = \xi(t, s, R) $.
In fact, we can compute
\begin{equation*}
\begin{aligned}
& 2 \vert \nabla u^{\xi}(x) \vert^2 + \kappa \vert \nabla v^{\xi}(x) \vert^2 \\
= & \left[ 2 \vert \nabla u(x) \vert^2 + \kappa \vert \nabla v(x) \vert^2 \right] - \left[ 4\kappa\xi\cdot\Im \left( u(x) \overline{\nabla u(x)} \right) + 2\kappa\xi\cdot\Im \left( v(x) \overline{\nabla v(x)} \right) \right] \\
& \quad + \left[ \kappa^2\vert \xi \vert^2 \vert u(x) \vert^2 + \vert \xi \vert^2 \vert v(x) \vert^2 \right], \\
& 2\kappa \vert u^{\xi}(y) \vert^2 + \vert v^{\xi}(y) \vert^2 = 2\kappa \vert u(y) \vert^2 + \vert v(y) \vert^2, \\
\text{and}& \\
& \Im \left( 2 u^{\xi}(x) \overline{\nabla u^{\xi}(x)} + v^{\xi}(x) \overline{\nabla v^{\xi}(x)} \right) \\
= & \Im \left( 2 u(x) \overline{\nabla u(x)} + v(x) \overline{\nabla v(x)} \right) - \left( 2\kappa\xi \vert u(x) \vert^2 + \kappa\xi \vert v(x) \vert^2 \right).
\end{aligned}
\end{equation*}
Thus,
\begin{equation*}
\begin{aligned}
& \left( 2 \vert \nabla u^{\xi}(x) \vert^2 + \kappa \vert \nabla v^{\xi}(x) \vert^2 \right) \left( 2\kappa \vert u^{\xi}(y) \vert^2 + \vert v^{\xi}(y) \vert^2 \right) \\
& \quad - \Im \left( 2 u^{\xi}(x) \overline{\nabla u^{\xi}(x)} + v^{\xi}(x) \overline{\nabla v^{\xi}(x)} \right) \Im \left( 2\kappa u^{\xi}(y) \overline{\nabla u^{\xi}(y)} + \kappa v^{\xi}(y) \overline{\nabla v^{\xi}(y)} \right) \\
= & \left( 2 \vert \nabla u(x) \vert^2 + \kappa \vert \nabla v(x) \vert^2 \right) \left( 2\kappa \vert u(y) \vert^2 + \vert v(y) \vert^2 \right) \\
& \quad - \Im \left( 2 u(x) \overline{\nabla u(x)} + v(x) \overline{\nabla v(x)} \right) \Im \left( 2\kappa u(y) \overline{\nabla u(y)} + \kappa v(y) \overline{\nabla v(y)} \right) \\
& \quad - \left( 2\kappa\xi \cdot \Im \left( u(x) \overline{\nabla u(x)} \right) + \kappa\xi \cdot \Im \left(  v(x) \overline{\nabla v(x)} \right) \right) \left( 2\kappa \vert u(y) \vert^2 + \vert v(y) \vert^2 \right) \\
& \quad + \left( 2\kappa\xi \vert u(x) \vert^2 + \xi \vert v(x) \vert^2 \right) \cdot \Im \left( 2\kappa u^{\xi}(y) \overline{\nabla u^{\xi}(y)} + \kappa v^{\xi}(y) \overline{\nabla v^{\xi}(y)} \right),
\end{aligned}
\end{equation*}
and hence the claim follows by symmetry of $ \Gamma^{2} $ and a change of variables.

The choice of $ \xi = \xi(t, s, R) $ in \eqref{xi} results in
\begin{equation*}
\int_{\mathbb{R}^5} \Im \left( 2 u(x) \overline{\nabla u(x)} + v(x) \overline{\nabla v(x)} \right) \Gamma^{2} \left(\frac{x-s}{R}\right) {\rm d}x = 0.
\end{equation*}
As a result,
\begin{equation}
\mathcal{C + E} = \frac{4}{\omega_{5}R^5} \int_{\mathbb{R}^5} \int_{\mathbb{R}^5} \int_{\mathbb{R}^5} L^{\xi}_{\kappa} \Gamma^{2} \left( \frac{x - s}{R} \right) \Gamma^{2} \left( \frac{y - s}{R} \right) \left( 2\kappa \vert u(y) \vert^2 + \vert v(y) \vert^2 \right) {\rm d}x {\rm d}y {\rm d}s,
\label{C+E}
\end{equation}
where $ L^{\xi}_{\kappa} = \left( 2 \vert \nabla u^{\xi}(x) \vert^2 + \kappa \vert \nabla v^{\xi}(x) \vert^2 \right) $.

Note that, by Cauchy-Schwartz inequality,
\begin{equation}
\Im \left( 2\kappa u \overline{\not\!\nabla u} + \kappa v \overline{\not\!\nabla v} \right) \le \sqrt{2\kappa \vert \not\!\nabla u \vert^2 + \kappa^2 \vert \not\!\nabla v \vert^2} \sqrt{2\kappa \vert u \vert^2 + \vert v \vert^2},
\label{CS}
\end{equation}
we have
\begin{equation}
\Im \left( 2 u \overline{\not\!\nabla u} + v \overline{\not\!\nabla v} \right) \Im \left( 2\kappa u \overline{\not\!\nabla u} + \kappa v \overline{\not\!\nabla v} \right) \le \left( 2 \vert \not\!\nabla u \vert^2 + \kappa \vert \not\!\nabla v \vert^2 \right) \left( 2\kappa \vert u \vert^2 + \vert v \vert^2 \right),
\label{D+F}
\end{equation}
which means $ \mathcal{D + F} \ge 0 $.

To conclude, we deduce
\begin{equation}
\begin{aligned}
\frac{\mathrm{d}}{\mathrm{d}t}M(t)
\ge & \int_{\mathbb{R}^5} \int_{\mathbb{R}^5} \Delta W_{\kappa} \left( 4 \psi(x-y) + \phi(x-y) \right) \left( \vert u(y) \vert^2 + \vert v(y) \vert^2 \right) {\rm d}x {\rm d}y \\
& -\frac{10}{\omega_{5}R^5} \int_{\mathbb{R}^5} \int_{\mathbb{R}^5} \int_{\mathbb{R}^5} \Re \left( \overline{v(x)} u^2(x) \right) \Gamma^{3} \left( \frac{x - s}{R} \right) \Gamma^{2} \left( \frac{y - s}{R} \right) N_{\kappa} {\rm d}x {\rm d}y {\rm d}s \\
& -8 \int_{\mathbb{R}^5} \int_{\mathbb{R}^5} \Re \left( \overline{v(x)} u^2(x) \right) \left( \psi(x-y) - \phi(x-y) \right) N_{\kappa} {\rm d}x {\rm d}y \\
& -10 \int_{\mathbb{R}^5} \int_{\mathbb{R}^5} \Re \left( \overline{v(x)} u^2(x) \right) \left( \phi(x-y) - \phi_{1}(x, y) \right) N_{\kappa} {\rm d}x {\rm d}y \\
& +\frac{4}{\omega_{5}R^5} \int_{\mathbb{R}^5} \int_{\mathbb{R}^5} \int_{\mathbb{R}^5} L^{\xi}_{\kappa} \Gamma^{2} \left( \frac{x - s}{R} \right) \Gamma^{2} \left( \frac{y - s}{R} \right) N_{\kappa} {\rm d}x {\rm d}y {\rm d}s \\
= & : \mathcal{A + G + H + I + J}.
\end{aligned}
\end{equation}

Next, we will average this inequality over $ t \in I $ and logarithmically over $ R \in [R_0, R_0e^{J}] $.

We start with $ \frac{\mathrm{d}}{\mathrm{d}t}M(t) $. Looking back at the definition of $ M(t) $, we find the upper bound $ \sup_{t \in \mathbb{R}} \vert M(t) \vert \lesssim RE_0^2 $. By the fundamental theorem of calculus, we have
\begin{equation}
\left\vert \frac{1}{T_0} \int_{I} \frac{1}{J} \int_{R_0}^{R_0e^{J}} \frac{\mathrm{d}}{\mathrm{d}t}M(t) \frac{{\rm d}R}{R} {\rm d}t \right\vert \lesssim \frac{1}{T_0} \frac{R_0e^{J}}{J} E_0^2.
\end{equation}

We turn to $ \mathcal{A} $ integrating by parts.
\begin{equation*}
\begin{aligned}
\mathcal{A} \ge - \int_{\mathbb{R}^5} \int_{\mathbb{R}^5} & \left( 4 \vert u(x) \vert \vert \nabla u(x) \vert + 2\kappa \vert v(x) \vert \vert \nabla v(x) \vert \right) \left\vert 4 \nabla\psi(x-y) + \nabla\phi(x-y) \right\vert \\
& \left( \vert u(y) \vert^2 + \vert v(y) \vert^2 \right) {\rm d}x {\rm d}y.
\end{aligned}
\end{equation*}
The facts $ \vert \nabla \phi \vert \lesssim \frac{1}{R} $ and $ \vert \nabla \psi \vert = \vert \frac{x}{\vert x \vert^2} (\phi - \psi) \vert \lesssim \min \{ \frac{1}{R}, \frac{R}{\vert x \vert^2} \} $ tell us
\begin{equation}
\frac{1}{T_0} \int_{I} \frac{1}{J} \int_{R_0}^{R_0e^{J}} \mathcal{A} \frac{{\rm d}R}{R} {\rm d}t \gtrsim - \frac{1}{JR_0} E_0^2.
\label{A}
\end{equation}

For $ \mathcal{G + J} $, we can establish a lower bound for these terms by Lemma \ref{coerb} after choosing $ \chi_{R}(x) = \Gamma \left( \frac{x-s}{R} \right) $, that is
\begin{equation}
\begin{aligned}
& \frac{1}{T_0} \int_{I} \frac{1}{J} \int_{R_0}^{R_0e^{J}} \mathcal{G + J} \frac{{\rm d}R}{R} {\rm d}t \\
\gtrsim & \frac{\delta}{JT_0} \int_{I} \int_{R_0}^{R_0e^{J}}\frac{1}{R^5} \int_{\mathbb{R}^5} \int_{\mathbb{R}^5} \int_{\mathbb{R}^5} L^{\xi}_{\kappa} \Gamma^{2} \left( \frac{x - s}{R} \right) \Gamma^{2} \left( \frac{y - s}{R} \right) N_{\kappa} {\rm d}x {\rm d}y {\rm d}s \frac{{\rm d}R}{R} {\rm d}t.
\end{aligned}
\label{G+J}
\end{equation}

As for $ \mathcal{H} $, by construction,
\begin{equation*}
\vert \psi(x) - \phi(x) \vert \lesssim \min \left\{ \frac{\vert x \vert}{R}, \frac{R}{\vert x \vert} \right\}.
\end{equation*}
We deduce
\begin{equation}
\frac{1}{T_0} \int_{I} \frac{1}{J} \int_{R_0}^{R_0e^{J}} \mathcal{H} \frac{{\rm d}R}{R} {\rm d}t \gtrsim - \frac{1}{J} E_0^2.
\label{H}
\end{equation}

Finally, similar to the estimates of $ \mathcal{H} $, we have
\begin{equation}
\frac{1}{T_0} \int_{I} \frac{1}{J} \int_{R_0}^{R_0e^{J}} \mathcal{I} \frac{{\rm d}R}{R} {\rm d}t \gtrsim - \eps E_0^2,
\label{I}
\end{equation}
because $ \vert \phi(x-y) - \phi_{1}(x, y) \vert \lesssim \eps $.

Collecting \eqref{A}, \eqref{G+J}, \eqref{H}, and \eqref{I}, we find
\begin{equation}
\begin{aligned}
& \frac{\delta}{JT_0} \int_{I} \int_{R_0}^{R_0e^{J}}\frac{1}{R^5} \int_{\mathbb{R}^5} \int_{\mathbb{R}^5} \int_{\mathbb{R}^5} L^{\xi}_{\kappa} \Gamma^{2} \left( \frac{x - s}{R} \right) \Gamma^{2} \left( \frac{y - s}{R} \right) N_{\kappa} {\rm d}x {\rm d}y {\rm d}s \frac{{\rm d}R}{R} {\rm d}t \\
\lesssim & \left( \frac{R_0e^{J}}{JT_0} + \frac{1}{JR_0} + \frac{1}{J} + \eps \right) E_0^2,
\end{aligned}
\end{equation}
which completes the proof of Proposition \ref{decay}.

\section{Proof of the main result}
In this section, we combine the results in Section 3 and 4 to complete the proof of Theorem \ref{main}.
More specifically, the result of interaction Morawetz estimate is used to verify the condition of scattering criterion.

\textbf{The proof of Theorem \ref{main}:}
First of all, using the rescaling
\[
\b{\rm u}_{\lambda}(t, x) = \lambda^2 \b{\rm u} \left( \lambda^2 t, \lambda x \right),
\]
we can always fix $ \lambda > 0 $ such that \eqref{E0} holds.
In order to establish \eqref{b1}, we change \eqref{Me} into
\[
\frac{\delta}{JT_0} \int_{I} \int_{R_0}^{R_0e^{J}} \frac{1}{R^5} \int_{\mathbb{R}^5} \int_{\mathbb{R}^5} \int_{\mathbb{R}^5} \mathcal{L}^{\xi}_{\kappa} \mathcal{N}_{\kappa} {\rm d}x {\rm d}y {\rm d}s \frac{{\rm d}R}{R} {\rm d}t \lesssim \left( \frac{R_0 e^{J}}{J T_0} + \eps \right) E_0^2,
\]
where
\[
\mathcal{L}^{\xi}_{\kappa} =  \left( 2 \left\vert \nabla \left( \Gamma \left( \frac{x-s}{R} \right) u^{\xi}(x) \right) \right\vert^2 + \kappa \left\vert \nabla \left( \Gamma \left( \frac{x-s}{R} \right) v^{\xi}(x) \right)  \right\vert^2 \right),
\]
and
\[
\mathcal{N}_{\kappa} = \left( 2\kappa \left\vert \Gamma^{2} \left( \frac{y - s}{R} \right) u(y) \right\vert^2 + \left\vert \Gamma^{2} \left( \frac{y - s}{R} \right) v(y) \right\vert^2 \right).
\]
If we choose $ J = \eps^{-1} R_0, T_0 = e^{J} $, then
\[
\frac{\delta}{JT_0} \int_{I} \int_{R_0}^{R_0e^{J}} \frac{1}{R^5} \int_{\mathbb{R}^5} \int_{\mathbb{R}^5} \int_{\mathbb{R}^5} \mathcal{L}^{\xi}_{\kappa} \mathcal{N}_{\kappa} {\rm d}x {\rm d}y {\rm d}s \frac{{\rm d}R}{R} {\rm d}t \lesssim_{E_0} \eps.
\]

Considering the support of $ \Gamma $ in \eqref{Gamma} and using the integral mean value theorem, we can omit writing some constants and get
\[
\frac{1}{T_0} \int_{I} \left\Vert \left( \Gamma u, \Gamma v \right) \right\Vert^2_{{\rm \dot{H}}_{x}^1(\mathbb{R}^5)} \left\Vert \left( \Gamma u, \Gamma v \right) \right\Vert^2_{{\rm L}_{y}^2(\mathbb{R}^5)} {\rm d}t
\lesssim_{E_0} \eps.
\]
The interval $ I $ is divided into many sub-intervals of the same length $ l = \eps^{-\frac{4}{5}} $.
By the pigeonhole principle, there exists $ I_{k} $ such that
\begin{equation*}
\int_{I_{k}} \left\Vert \left( \Gamma u, \Gamma v \right) \right\Vert^2_{{\rm \dot{H}}_{x}^1(\mathbb{R}^5)} \left\Vert \left( \Gamma u, \Gamma v \right) \right\Vert^2_{{\rm L}_{y}^2(\mathbb{R}^5)} {\rm d}t
\lesssim_{E_0} \eps^{\frac{1}{5}},
\end{equation*}
which means
\begin{equation*}
\int_{I_{k}} \left\Vert \left( \Gamma u, \Gamma v \right) \right\Vert^4_{{\rm L}^{\frac{5}{2}}(\mathbb{R}^5)} {\rm d}t \lesssim_{E_0} \eps^{\frac{1}{5}}.
\end{equation*}
At the same time, thanks to \eqref{E0}, we have
\begin{equation*}
\left\Vert \left( \Gamma u, \Gamma v \right) \right\Vert^2_{{\rm L}_{t}^2(I_{k}, {\rm L}_{x}^{\frac{5}{2}}(\mathbb{R}^5))}
= \int_{I_{k}} \left\Vert \left( \Gamma u, \Gamma v \right) \right\Vert^2_{{\rm L}^{\frac{5}{2}}(\mathbb{R}^5)} {\rm d}t \lesssim_{E_0} \eps^{\frac{1}{5}}.
\end{equation*}

By interpolation,
\begin{equation*}
\left\Vert \left( \Gamma u, \Gamma v \right) \right\Vert_{{\rm L}_{t}^6(I_{k}, {\rm L}_{x}^{3}(\mathbb{R}^5))}
\le \left\Vert \left( \Gamma u, \Gamma v \right) \right\Vert^{\frac{1}{3}}_{{\rm L}_{t}^2(I_{k}, {\rm L}_{x}^{\frac{5}{2}}(\mathbb{R}^5))}
\left\Vert \left( \Gamma u, \Gamma v \right) \right\Vert^{\frac{2}{3}}_{{\rm L}_{t}^{\infty}(I_{k}, {\rm H}_{x}^1(\mathbb{R}^5))}
\lesssim_{E_0} \eps^{\frac{1}{30}},
\end{equation*}
which implies \eqref{b1} holds.
So the scattering criterion, Proposition \ref{scat}, tells us that $ \b{\rm u} $ scatters, which completes the proof of Theorem \ref{main}.


\begin{thebibliography}{10}
\bibitem{Bourgain1999} J. Bourgain, Global wellposedness of defocusing critical nonlinear Schr\"odinger equation in the radial case, J. Amer. Math. Soc., 12(1999), 145-171.
\bibitem{CK} M. Christ and A. Kiselev, Maximal functions associated to filtrations. J. Funct. Anal., 179 (2001), 409--425.
\bibitem{Colin2009} M. Colin, Th. Colin and M. Ohta, Stability of solitary waves for a system of nonlinear Schr\"odinger equations with three wave interaction. Ann. Inst. H. Poincar\'{e} Anal. Non Lin\'{e}aire 26 (2009), no. 6, 2211-2226.
\bibitem{Colliander2008} J. Colliander, M. Keel, G. Staffilani, H. Takaoka, and T. Tao, Global well-posedness and scattering for the energy-critical nonlinear Schr\"odinger equation in $ \mathbb{R}^{3} $, Annals of Math., 167(2008), 767-865.
\bibitem{Dodson2012} B. Dodson, Global well-posedness and scattering for the defocusing, $ L^{2} $-critical, nonlinear Schr\"odinger equation when $ d \geqslant 3 $, J. Amer. Math. Soc., 25(2012), 429-463.
\bibitem{Dodson} B. Dodson, Global well-posedness and scattering for the defocusing, $ L^{2} $-critical, nonlinear Schr\"odinger equation when $ d = 2 $, Duke Math. J. 165, 18(2016), 3435-3516.
\bibitem{Dodson20161} B. Dodson, Global well-posedness and scattering for the defocusing, $ L^{2} $-critical, nonlinear Schr\"odinger equation when $ d = 1 $, J. Amer. Math., no. 2, 138(2016), 531-569.
\bibitem{Dodson2017} B. Dodson, C. Miao, J. Murphy and J. Zheng, The defocusing quintic NLS in four space dimensions, Ann. Inst. Henri Poincar\'{e}-AN, 34, (2017), no. 2, 759-787.
\bibitem{Dodson2016} B. Dodson and J. Murphy, A new proof of scattering below the ground state for the 3D radial focusing cubic NLS, Proc. Amer. Math. Soc. 145 (2017), no. 11, 4859-4867.
\bibitem{Dodson2018} B. Dodson and J. Murphy, A new proof of scattering below the ground state for the non-radial focusing NLS, Math. Res. Lett. 25 (2018), no. 6, 1805-1825.
\bibitem{Duyckaerts2007} T. Duyckaerts, J. Holmer and S. Roudenko, Scattering for the non-radial 3D cubic nonlinear Schr\"odinger equation. Math. Res. Lrtt. 15(2008), 1233-1250.
\bibitem{Ginibre1992} J. Ginibre and G. Velo, Smoothing properties and retarded estimates for some dispersive evolution equations. Comm. Math. Phys., 144 (1992), no. 1, 163-188.
\bibitem{Hamano2018} M. Hanamo, Global dynamics below the ground state for the quadratic Schr\"odinger system in 5d, preprint, arXiv:1805.12245, 2018.
\bibitem{Hamano2019} M. Hanamo, T. Inui and K. Nishimura, Scattering for the quadratic nonlinear Schr\"odinger system in $ \mathbb{R}^5 $ without mass-resonance condition, preprint, arXiv:1903.05880, 2019.
\bibitem{Keel1998} M. Keel and T. Tao, Endpoint Strichartz estimate. Amer. J. Math., 120 (1998), no. 5, 955-980.
\bibitem{Kenig2006} C. E. Kenig and F. Merle, Global well-posedness, scattering and blow-up for the energy-critical, focusing, non-linear Schr\"odinger equation in the radial case, Invent. Math., 166(2006), no. 3, 645-675.
\bibitem{Killip2010} R. Killip and M. Visan, The focusing energy-critical nonlinear Schr\"odinger equation in dimensions five and higher. Amer. J. Math., 132 (2010), 361-424.
\bibitem{Miao2013} C. Miao, G. Xu, L. Zhao, The dynamics of the 3D radial NLS with the combined terms, Commun. Math. Phys., 318, (2013), no. 1, 767-808.
\bibitem{Miao2014} C. Miao, J. Murphy, J. Zheng, The defocusing energy-supercritical NLS in four space dimensions, J. Functional Analysis, 267, (2014), no. 2, 1662-1724.
\bibitem{Miao2017} C. Miao, T. Zhao, J. Zheng, On the 4D nonlinear Schr\"odinger equation with combined terms under the energy threshold, Calculus of Variations and PDE, 56-179, (2017), no. 1, 1-39.
\bibitem{Strichartz1977} R. Strichartz, Restriction of Fourier transform to quadratic surfaces and decay of solutions of wave equations. Duck Math. J. 44 (1977), 705-774.
\bibitem{Tao2004} T. Tao, On the asymptotic behavior of large radial data for a focusing non-linear Schr\"odinger equation. Dyn. Partial Differ. Equ., 1(2004), no. 1, 1-48.
\bibitem{Visan2007} M. Visan, The defocusing enery-critical nonlinear Schr\"odinger equation in higher dimensions, Duck Math. J., 138(2007), 281-374. 
\bibitem{Xia2019} S. Xia and C. Xu, On dynamics of the system of two coupled nonlinear Schr\"odinger in $ \mathbb{R}^3 $, http://doi.org/10.1002/mma.5814. 
\bibitem{Xu2014} G. Xu, Dynamics of some coupled nonlinear Schr\"odinger systems in $ \mathbb{R}^3 $, Math. Methods Appl. Sci., 37(2014), no. 17, 2746-2771.
\end{thebibliography}
\end{document}